\documentclass[12pt]{amsart}

\usepackage{cancel}
\usepackage{latexsym, amssymb, amsmath}
\usepackage{soul,esint}
\usepackage{amsfonts, graphicx}
\usepackage{graphicx,color}
\usepackage[backend=biber,style=alphabetic,sorting=nyt]{biblatex}
\usepackage{amsthm}
\usepackage{hyperref}
\usepackage{enumerate}
\usepackage{color}
\usepackage{verbatim}
\def\e{\varepsilon}

\def\a{{\alpha}}

\def\tr{\operatorname{tr}}

\theoremstyle{plain}
\newtheorem{theorem}[subsection]{Theorem}
\theoremstyle{plain}
\newtheorem{proposition}[subsection]{Proposition}
\newtheorem{lemma}[subsection]{Lemma}
\newtheorem{corollary}[subsection]{Corollary}

\theoremstyle{definition}
\newtheorem{definition}[subsection]{Definition}

\theoremstyle{remark}
\newtheorem{remark}[subsection]{Remark}

\numberwithin{equation}{subsection}

\newcommand{\del}{\nabla}
\newcommand{\hdel}{\tilde{\nabla}}

\begin{document}
\title[]
{Ricci-DeTurck Flow from Initial Metric with Morrey-type Integrability Condition}

\author{Man-Chun Lee}
\address[Man-Chun Lee]{Room 237, Lady Shaw Building,
The Chinese University of Hong Kong,
Shatin, N.T., Hong Kong}
\email{mclee@math.cuhk.edu.hk}

\author{Stephen Shang Yi Liu}
\address[Stephen Shang Yi Liu]{Room 3446, Academic Building, The Hong Kong University of Science and Technology, Clearwater Bay, Kowloon, Hong Kong}
 \email{masyliu@ust.hk}

\renewcommand{\subjclassname}{
  \textup{2010} Mathematics Subject Classification}

\date{\today}

\begin{abstract}
  In this work, we study the short-time existence theory of Ricci-DeTurck flow starting from rough metrics which satisfy a Morrey-type integrability condition. Using the rough existence theory, we show the preservation and improvement of distributional scalar curvature lower bounds  provided the singular set for such metrics is not too large. As an application, we use the Ricci flow smoothing to study the removable singularity related to scalar curvature. Our result supplements those of Jiang-Sheng-Zhang.
\end{abstract}


\maketitle

%
%
%

\section{Introduction}\label{sec:introduction}

Let $M^n$ be a Riemannian manifold. Then the Ricci flow is a family of metrics $\hat g(t)$ evolving in the direction of their Ricci tensors,
\begin{equation*}
    \frac{\partial}{\partial t} \hat g(t) = -2\text{Ric}(\hat g(t)).
\end{equation*}

Introduced by Hamilton in \cite{hamilton_three-manifolds_1982}, the Ricci flow has seen a number of successful applications to a number of problems in geometry, most famously in Perelman's resolution of Thurston's Geometrization conjecture. It is a weakly parabolic system, and in \cite{deturck_deforming_1983}, DeTurck showed that it is diffeomorphic to a strictly parabolic system, the Ricci-DeTurck $h$-flow. Precisely, Let $\text{Sym}_2(T^\ast M)$ denote the space of symmetric 2-tensors on $M$ and let $X:\text{Sym}_2(T^\ast M) \to TM$ be given by
\begin{equation*}
    X_{h}(g)^k := g^{ij}\left(\tilde{\Gamma}^k_{ij} - \Gamma^k_{ij}\right)
\end{equation*}
where $\tilde{\Gamma}$ and $\Gamma$ are the Christoffel symbols of $h$ and $g$ respectively. Then the Ricci-DeTurck $h$-flow (here using the terminology as in \cite{simon_deformation_2002} to emphasize the dependence on the background metric $h$) is given by the equation
\begin{equation}\label{eqn:prelim-ricci-deturck-flow}
    \frac{\partial}{\partial t} g(t) = -2\text{Ric}(g(t)) - \mathcal{L}_{X_{h}(g(t))}g(t)
\end{equation}
Solutions to the Ricci-DeTurck $h$-flow are related to a Ricci flow via pullback by diffeomorphisms, that is, if $g(t)$ solves (\ref{eqn:prelim-ricci-deturck-flow}) and $\chi_t:M\to M$ is a family of diffeomorphisms satisfying
\begin{equation}\label{eqn:prelim-rdf-to-rf-diffeomorphisms}
    \begin{cases}
        &X_{h(t)}(g(t))f = \frac{\partial}{\partial t}\left(f \circ \chi_t\right) \text{ for all } f \in C^\infty(M) \\
        &\chi_0 = \text{id},
    \end{cases}
\end{equation}
then $\chi_t^\ast g(t)$ is a Ricci flow with initial data $g(0)$.

In local coordinates, (\ref{eqn:prelim-ricci-deturck-flow}) is given by
\begin{align}\label{eqn:prelim-rdf-local-coords}
    \frac{\partial}{\partial t} g_{ij} &= g^{pq}\hdel_p\hdel_qg_{ij} - g^{k\ell}g_{ip}h^{pq}\tilde{R}_{jkq\ell} - g^{k\ell}g_{jp}h^{pq}\tilde{R}_{ikq\ell} \nonumber \\
    &+ \frac{1}{2}g^{k\ell}g^{pq}\left(\hdel_i g_{pk}\hdel_j g_{q\ell} + 2\hdel_k g_{jp}\hdel_q g_{i\ell} - 2\hdel_k g_{jp} \hdel_\ell g_{iq} \right. \nonumber \\
    &\left. -2\hdel_j g_{pk} \hdel_\ell g_{iq} - 2\hdel_i g_{pk} \hdel_\ell g_{qj}\right),
\end{align}
where $\hdel$ and $\tilde{R}$ denote the Levi-Civita connection and Riemann curvature tensor respectively of $h$ and we are suppressing the dependence on $t$ for $g(t)$ in the notation above.

Since the flow is a parabolic system, one usually expects that the flow will improve the regularity of initial data. In view of this, there has been a wide variety of literature extending the existence theory of Ricci Flow and Ricci-DeTurck $h$-flow to settings with less regularity. The foundational result in this area is by Shi in \cite{shi_deforming_1989} who showed the short-time existence of a solution to the Ricci-DeTurck flow from a metric that is complete with bounded curvature. A metric $g$ is called $L^\infty$ or bi-Lipschitz if $\Lambda^{-1}h \leq g \leq \Lambda h$ almost everywhere on $M$ with $\Lambda >1$ and $h$ is some fixed smooth background metric. Building on the result of Shi, Simon in \cite{simon_deformation_2002} established the following regularization result: if $(1+\delta)^{-1}h \leq g_0 \leq (1+\delta)h$ for some sufficiently small dimensional constant $\delta$, then there exists a short-time solution to the Ricci-DeTurck $h$-flow on $M \times (0,T]$ which is smooth for $t > 0$ with convergence back to $g_0$ in $C^0_\text{loc}$ if $g_0$ was additionally in $C^0_\text{loc}$. In Euclidean space, Koch and Lamm in \cite{koch_geometric_2012} used heat kernel estimates to show the existence and uniqueness of a global and analytic solution to the Ricci-DeTurck $h$-flow on $\mathbb{R}^n$ from $L^\infty$ initial metrics that are close in $L^\infty$ to the Euclidean metric. Building on these results, Burkhardt-Guim in \cite{burkhardt-guim_pointwise_2019} further studied Ricci flow from $C^0$ initial metrics on closed manifolds, in particular relying on heat kernel estimates to establish an iteration scheme to construct a short-time solution and applied this Ricci flow existence theory to study scalar curvature lower bounds and a question of Gromov. In the direction of more general $L^\infty$ initial data, Lamm and Simon in \cite{lamm_ricci_2021} showed the short-time existence for the Ricci-DeTurck $h$-flow on complete four-manifolds for rough initial metrics that are in $L^\infty \cap W^{2,2}$. They showed that solutions are uniformly smooth for positive time and converge back to the initial data in $W^{2,2}_\text{loc}$ sense as $t\to0$. In \cite{chu_ricci-deturck_2022}, Chu and the first named author developed this further and considered the case where $g_0 \in L^\infty \cap W^{1,n}$, showing the existence of the flow starting from metrics which are $L^\infty$ and satisfy a small local gradient concentration in a $W^{1,n}$ sense, and study a number of applications.

Motivated by the work \cite{chu_ricci-deturck_2022}, and the heat kernel based approach of Burkhardt-Guim \cite{burkhardt-guim_pointwise_2019}, in this work we consider the existence theory for metrics that are in $L^\infty$ and satisfy a $L^1$ Morrey-type integrability condition (condition (iii) of Theorem \ref{thm:intro-application-1-statement} below) and study a removable singularity theorem related to scalar curvature. We will see that the Morrey-type condition allows for the use of heat kernel estimates to establish a priori estimates that are important for the existence theory. Our main theorem is the following existence theorem for the smooth $h$-flow with quantitative estimates.

\begin{theorem}\label{thm:intro-application-1-statement}
    Suppose $(M^n, h)$ is a complete Riemannian manifold satisfying \eqref{eqn:h-remark-curvature-estimates}. Suppose $g_0$ is a $C^0_{loc}\cap W^{1,2}_{loc}$ Riemannian metric on $M$ and $\Sigma \subseteq M$ is a compact set so that the following holds:
    \begin{enumerate}[(i)]
        \item $g_0$ is smooth on $M \setminus \Sigma$;
        \item $g_0$ is globally uniformly bi-Lipschitz on $M$: 
         $\exists\Lambda_0 > 1$ such that $$\Lambda_0^{-1} h \leq g_0 \leq \Lambda_0 h\;\;\quad\text{on}\;\;M;$$
        \item there exist $L_0,\delta,r_0 > 0$ and $p\geq 1$ such that for every $x_0 \in M$, $0 < r < r_0$,
        \begin{equation*}
            \fint_{B(x_0,r)} |\hdel g_0|^p d\mu_h \leq L_0 r^{-p+\delta}.
        \end{equation*}
    \end{enumerate}
    Then there are $T,L>0$ depending only on  $n, \Lambda_0, L_0, p, \delta, r_0,h$ and a smooth solution $g(t)$ to the Ricci-DeTurck $h$-flow on $M\times(0,T]$ such that:
        \begin{enumerate}[(a)]
            \item for any $k \in \mathbb{N}$, there is $B_k(n, k,r_0, \Lambda_0, L_0, p, \delta,h) > 0$ such that for all $t \in (0, T]$, $x_0 \in M$,
            \begin{equation*}
                \sup\limits_{M}|\hdel^k g(t)| \leq \frac{B_k}{t^{\frac{1}{2}(k-\frac{\delta}{p})}};
            \end{equation*}
            \item $g(t) \to g_0$ in $C_\text{loc}^0(M)$ as $t \to 0$;
            \item $g(t) \to g_0$ in $C_\text{loc}^\infty(M\setminus\Sigma)$ as $t \to 0$;
        \end{enumerate}
        Moreover, the solution is unique within the class of solutions satisfying (a).
\end{theorem}

The result is new even if $\Sigma=\emptyset$. Our work is motivated by the removable singularity in view of scalar curvature. In \cite{lee_positive_2015}, Lee--Lefloch introduced a notion of distributional scalar curvature for metrics with lower regularity. In \cite{jiang_removable_2022}, the authors showed that for metrics $g \in C^0 \cap W^{1,p}_{\text{loc}}(M)$ where $M$ is complete asymptotically flat with $n \leq p \leq \infty$ satisfying $R_g \geq 0$ outside some compact singular set $\Sigma$ also has non-negative ADM mass (with rigidity) provided the singular set is not too large in terms of Hausdorff measure. In the compact case,  the authors in \cite{jiang_weak_2021} also show that distributional scalar curvature lower bounds can be preserved and improved to classical scalar curvature lower bounds along Ricci flow for metrics in $W^{1,p}(M)$ for $n < p \leq \infty$. These results are related to positive mass theorems for lower regularity metrics and a conjecture of Schoen regarding the removable singularity of $L^\infty$ metrics on $\mathbb{T}^n$ with co-dimension three singularity and $R_g \geq 0$ on the regular part. In this work, we are primarily interested in the following question: if we strengthen the regularity from $C^0$ slightly to the Morrey-type condition on the gradient, then under what conditions on the singular set could we expect to have scalar curvature rigidity. In Theorem \ref{thm:preservation-distributional-scalar-curvature-lower-bound} and Corollary \ref{thm:non-cpt-dist-scalar-curvature-lower-bdd}, we use the rough existence theory established in Theorem \ref{thm:intro-application-1-statement} to obtain similar results for metrics which satisfy the Morrey-type condition, as long as the singular set is small in terms of a notion of co-dimension considered in \cite{lee_continuous_2021}.

We now outline the rest of this work. In Section \ref{sec:a-priori-estimates}, we establish the a priori estimates that will be used later. In Section \ref{sec:proof-of-main-theorem}, we prove Theorem \ref{thm:intro-application-1-statement}. Finally, in Section \ref{sec:applications}, we study the application outlined above.

{\it Acknowledgement}:
The first-named author was partially supported by Hong Kong RGC grant (Early Career Scheme) of Hong Kong No. 24304222, No. 14300623, and a NSFC grant No. 12222122.

\section{a-priori estimates}\label{sec:a-priori-estimates}

In this section, we establish quantitative estimates that will be important in the proof of the main theorem. In general, we consider rough metrics $g_0$ inside a smooth background manifold $(M,h)$ with $\sup_M|\mathrm{Rm}(h)|<+\infty$. Thanks to Shi's smoothing result \cite{shi_deforming_1989}, in our content we will as well assume  
\begin{equation}\label{eqn:h-remark-curvature-estimates}
  \sup\limits_{M}|\hdel^i\text{Rm}(h)|:=k_i<+\infty,\;\; \forall i\in \mathbb{N}
\end{equation}
where $\hdel$ denotes the connection of $h$. In the following and the rest of this work, we denote by $k_i := \sup\limits_M|\hdel^i\text{Rm}(h)|$ and we will specify using parentheses the quantities that certain numbered constants depend upon, e.g. $C_j = C_j(n,k_1,\dots,k_i) > 0$ denotes a positive constant $C_j$ that depends on the dimension and quantities $k_1$, up to $k_i$. These constants may be changing line by line throughout and we will only re-introduce the parentheses when the dependency changes. All distances, norms and connection are measured with respect to the metric $h$, unless noted otherwise. We also use $a\wedge b$ to denote $\min\{a,b\}$.

\subsection{Parabolic boot-strapping}

Following \cite{shi_deforming_1989} and \cite{simon_deformation_2002}, the following lemma says that given a first order interior estimate, we may parabolically bootstrap to obtain higher order interior estimates.

\begin{lemma}\label{lem:estimates-bootstrap}
    Suppose $g(t)$ is a smooth solution to the Ricci-DeTurck $h$-flow on $M \times [0,T]$ for some smooth background metric $h$ satisfying (\ref{eqn:h-remark-curvature-estimates}) 
    so that
    \begin{enumerate}[(i)]
        \item $\Lambda^{-1}h \leq g(t) \leq \Lambda h$ on $M\times[0,T]$ for some $\Lambda > 1$;
        \item there exists $x_0\in M$, $B_1,\gamma>0$ such that 
        \begin{equation*}
            |\hdel g(x,t)| \leq \frac{B_1}{t^{\frac{1}{2}-\frac{\delta}{2p}}}
        \end{equation*}
        for all $x \in B\left(x_0, 1 + \frac{\gamma}{2}\right), t \in (0, T]$. 
    \end{enumerate}
 Then for any $m\in \mathbb{N}$, there are 
       $ B_m(n,\gamma, m,\Lambda,k_1,\dots,k_m) > 0$
    such that 
    \begin{equation}\label{eqn:bootstrap-higher-order-interior-estimate}
        \sup\limits_{B\left(x_0, 1 + \frac{\gamma}{2^{m+1}}\right)} |\hdel^m g(x,t)| \leq \frac{B_m}{t^{\frac{1}{2}(m-\frac{\delta}{p})}}
    \end{equation}
    for all $t \in (0, T\wedge 1]$. 
\end{lemma}

\begin{proof}[Sketch of Proof]
    The proof of the above Lemma is a standard induction argument, using Bernstein-Shi's trick and the non-scaling invariant order on $t$ given in assumption (ii) above instead of the $t^{-\frac{1}{2}}$ found in Lemmas 4.1-4.2 of \cite{simon_deformation_2002}. For $m \in \mathbb{N}$, we obtain (\ref{eqn:bootstrap-higher-order-interior-estimate}) by considering the evolution equation under the operator $\partial_t - g^{ij}\hdel_i\hdel_j$ of 
    \begin{equation*}
        t^{2m-1-\frac{2\delta}{p}}|\hdel^m g|^2\left(|\hdel^{m-1}g|^2 + Lt^{-m+1+\frac{\delta}{p}}\right),
    \end{equation*}
  together with standard cut-off argument.
\end{proof}

\subsection{Gradient Estimate}

We now show that the interior gradient estimate assumed in Lemma \ref{lem:estimates-bootstrap} can be obtained for smooth solutions of the Ricci-DeTurck $h$-flow under the $L^\infty$ and Morrey-type integrability assumption. Once the $C^1$ estimate has been established, Lemma \ref{lem:estimates-bootstrap} gives us all higher order estimates.

\begin{proposition}\label{prop:gradient-estimate}
    Suppose $g(t)$ is a smooth solution to the Ricci-DeTurck $h$-flow on $M \times [0,T]$ for some smooth initial background metric $h$ satisfying (\ref{eqn:h-remark-curvature-estimates}) so that
    \begin{enumerate}[(i)]
        \item $\Lambda^{-1}h \leq g(t) \leq \Lambda h$ on $M\times[0,T]$ for some $\Lambda > 1$;
        \item there exists $p \geq 1, L_0, r_0 ,\delta > 0$,  such that for all $x_0 \in M$ and $0 < r < r_0$,
        \begin{equation*}
            \fint_{B(x_0,r)} |\hdel g(x,0)|^p d\mu_h(x) \leq L_0r^{-p+\delta};
        \end{equation*}
        \item $\sup_{M\times [0,T]} |\tilde\nabla g(x,t)|<+\infty$.
    \end{enumerate}
    Then there are $B_1 = B_1(n,\Lambda,L_0,k_1,k_2), T_1 = T_1(n,\Lambda,L_0,p,\delta,k_1,k_2) > 0$ such that
    \begin{equation}\label{eqn:interior-gradient-estimate}
        \sup\limits_M|\hdel g(x,t)| \leq \frac{B_1}{t^{\frac{1}{2}-\frac{\delta}{2p}}}
    \end{equation}
    for all $x \in M, t \in (0, T_1\wedge T]$.
\end{proposition}

\begin{proof}
    For the sake of convenience, we will suppress the dependence on $t$ and denote $g(t)$ by $g$ unless explicitly stated otherwise. We might assume $T\leq 1,p>\delta$ and $r_0\leq 1$. We will use $C$ to denote any constants depending only on $n,k_1,k_2,\Lambda,L_0$ which might change line by line.
    
Let $B_1\geq 1$ be a constant to be chosen. We let $T_1$ to be the maximal time on which (\ref{eqn:interior-gradient-estimate}) holds. We have $T_1>0$ thanks to (iii). Moreover, if $T_1<T$, then  there is a $x_0 \in M$ such that
    \begin{equation}\label{eqn:gradient-estimate-maximal-T1}
        \begin{cases}
            &\sup\limits_M |\hdel g(x,t)| \leq \frac{B_1}{t^{\frac{1}{2}-\frac{\delta}{2p}}} \text{ on } (0,T_1) \text{ and, } \\
            &|\hdel g(x_0, T_1)| \geq   \frac{B_1}{2T_1^{\frac{1}{2}-\frac{\delta}{2p}}}.
        \end{cases}
    \end{equation}
   It suffices to estimate $T_1$ from below. We might assume $T_1\leq  T$ since otherwise the result holds trivially. 
    
    Our first step is to compute the evolution equation of $|\hdel g|^2$ under the operator $\Box_g := \frac{\partial}{\partial t} - \Delta^{g} - \nabla_X$ where $\Delta^g = g^{ij}\del_i\del_j$ is taking derivatives with respect to the metric $g$ instead of $h$. By (\ref{eqn:prelim-ricci-deturck-flow}) and Young's inequality, 
    \begin{align}\label{eqn:gradient-estimate-partialt-norm-hdelg}
        \frac{\partial}{\partial t} |\hdel g|^2 \leq g^{ij}\hdel_i\hdel_j|\hdel g|^2 - C^{-1}|\hdel^2 g|^2 + C |\hdel g|^4 + C|\hdel g|.
    \end{align}
    Since $|\hdel g|$ may not be smooth, define the smooth function $Q$ by
    \begin{equation*}
        Q := \left(|\hdel g|^2 + 1\right)^{\frac{1}{2}}
    \end{equation*}
%
which satisfies
    \begin{equation}\label{eqn:gradient-estimate-box-Q2}
        \Box_g Q \leq C Q^3 + C.
    \end{equation}
    Let $f(t) = \exp\left(-A(t^{\frac{\delta}{p}} +  t)\right)$ for a constant $A > 0$ to be determined. For all $t < T_1$, (\ref{eqn:gradient-estimate-box-Q2}) and (\ref{eqn:gradient-estimate-maximal-T1}) imply 
    \begin{align}\label{eqn:gradient-estimate-box-Q3}
        \Box_g\left(f(t)Q-Ct\right) &= f(t)\left(-\frac{A\delta}{p}t^{\frac{\delta}{p}-1} - A\right)Q + f(t)\Box_g Q - C \nonumber \\
        & \leq 0
    \end{align}
    by choosing $A$ sufficiently large in the second inequality above and using the fact that $f(t) \leq 1$ for all $t \geq 0$ for the last inequality.
    Hence, by the maximum principle we obtain
    \begin{equation}\label{eqn:gradient-estimate-invoke-heat-kernel}
        \frac14 f(T_1)B_1T_1^{-\frac{1}{2}+\frac{\delta}{2p}} - CT_1 \leq \int_{M} K(x_0, T_1; y, 0)f(0)Q(y,0)d\mu_h(y)
    \end{equation}
    where $K$ is the heat kernel for the heat operator $\Box_g$.
    
    We now wish to obtain an estimate for the heat kernel $K$. We do this by modifying an estimate established by Bamler--Cabezas-Rivas--Wilking in \cite{bamler_ricci_2017}. Let $\widetilde{G}(x,t;y,s)$ denote the heat kernel associated to the backwards heat equation coupled with a Ricci flow $h(t)$. That is,
    \begin{equation*}
        \left(\partial_s + \Delta^{h_t}_{y,s}\right)\widetilde{G}(x,t;\cdot,\cdot) = 0,\quad \text{and} \quad \lim\limits_{s\to t^-}\widetilde{G}(x,t;\cdot,s) = \delta_x(\cdot).
    \end{equation*}
    Then for all $(y,s) \in M \times [0,T]$, $\widetilde{G}(\cdot,\cdot;y,s)$ is the heat kernel associated to the conjugate equation
    \begin{equation*}
        \left(\partial_t - \Delta^{h_t}_{x,t} - R_{h_t}\right)\widetilde{G}(\cdot,\cdot;y,s) = 0,\quad \text{and} \quad \lim\limits_{t\to s^+}\widetilde{G}(\cdot,t;y,s) = \delta_y(x).
    \end{equation*}
    Proposition 3.1 of \cite{bamler_ricci_2017} gives the following estimate: for any $n, A > 0$, there is a $C(n,A) < \infty$ such that the following holds: Let $(M, h(t)), t\in[0,T]$ be a complete Ricci flow satisfying
    \begin{equation}\label{eqn:BCW-estimate-heat-kernel-assumptions}
        |\text{Rm}(x,t)| \leq At^{-1},\quad \text{and} \quad \text{Vol}_{h_t}\left(B_{h_t}(x,\sqrt{t})\right) \geq A^{-1}t^{n/2}
    \end{equation}
    for all $(x,t)\in M\times(0,T]$. Then
    \begin{equation}\label{eqn:BCW-estimate}
        \widetilde{G}(x,t;y,s) \leq \frac{C}{t^{n/2}}\exp{\left(-\frac{d^2_s(x,y)}{Ct}\right)}
    \end{equation}
    for any $0 \leq 2s \leq t$. 
    
    Let $G(x,t;y,s)$ denote the heat kernel associated with the heat equation coupled with the Ricci Flow $h(t)$
    \begin{equation*}
        \left(\partial_t - \Delta^{h_t}_{x,t}\right)G(\cdot,\cdot;y,s) = 0,\quad \text{and}\quad \lim\limits_{t\to s^+}G(\cdot,t;y,s) = \delta_y(x).
    \end{equation*}
    Choosing the Ricci flow $h(t)$ above to be the one related to $g(t)$ via pullback by diffeomorphisms generated by $X$, that is, $h(t) = \chi_t^\ast g(t)$, then the heat kernel $K$ that we are interested is related to the heat kernel $G$ also by pulling back the diffeomorphisms generated by $X$. We first show that under our assumptions, (\ref{eqn:BCW-estimate}) can be used to obtain an estimate for $G(x,t;y,0)$. By our assumptions and Lemma \ref{lem:estimates-bootstrap}, we have the bound 
    \begin{equation*}
        \sup_M|R_{h_t}|=\sup_M|R_{g_t}| \leq B_2 t^{-1+\alpha}
    \end{equation*}
    for some constant $B_2(n,\Lambda,k_1,k_2,B_1) < \infty$ and some small $\alpha = \alpha(p,\delta) > 0$. Then we have
    \begin{equation*}
        \left(\partial_t - \Delta^{h_t}_{x,t}\right)\widetilde{G} = R_{h_t}\widetilde{G} \geq -B_2t^{\alpha-1}\widetilde{G}
    \end{equation*}
    which gives
    \begin{align*}
      &\quad  \left(\partial_t - \Delta^{h_t}_{x,t}\right)(\exp{(\a^{-1}B_2t^\alpha)}\widetilde{G})\\ &\geq -B_2t^{\alpha - 1}\exp{(\a^{-1}B_2t^\alpha)}\widetilde{G} +  B_2t^{\alpha-1}\exp{(\a^{-1}B_2t^\alpha)}\widetilde{G}\geq 0.
    \end{align*}
Since at time $t = s>0$, $G(x,t;y,s) = \widetilde{G}(x,t;y,s)=\delta_y(x)$, the maximum principle  gives
    \begin{equation*}
        G(x,t;y,s)\leq \exp{(\a^{-1}B_2t^\alpha)}\widetilde{G}(x,t;y,s) \leq 2\widetilde{G}(x,t;y,s)
    \end{equation*}
 if we shrink $T_1$ depending on $B_2,\delta,p$. Note that our assumptions also ensures that the Ricci flow $h(t) = \chi_t^\ast g(t)$ satisfies the assumptions (\ref{eqn:BCW-estimate-heat-kernel-assumptions}) above. Finally, writing $K(x,t;y,s) = G(\chi_t^{-1}(x),t;\chi_s^{-1}(y),s)$ and after computations similar to that of Lemma 2.9 of \cite{burkhardt-guim_pointwise_2019}, we obtain the following estimate for $K$: there is a $C = C(n,\Lambda) > 0$ such that
    \begin{equation}\label{eqn:gradient-estimate-K-heat-kernel-estimate}
        K(x,t;y,s) \leq \frac{C}{t^{n/2}}\exp{\left(-\frac{d^2_h(x,y)}{Ct}\right)}
    \end{equation}
    for any $t\in (0,  T]$.
    
    By \eqref{eqn:gradient-estimate-K-heat-kernel-estimate}, the co-area formula and Stokes' Theorem, we have on the right hand side of (\ref{eqn:gradient-estimate-invoke-heat-kernel})
    \begin{align}\label{eqn:gradient-estimate-Q-heat-kernel-estimate2}
        &\int_{M} K(x_0, T_1; y, 0)f(0)Q(y,0)d\mu_h(y) \nonumber \\
        \quad& \leq \int_{M} \frac{C_0}{T_1^{\frac{n}{2}}}\exp\left(-\frac{d^2_{h_0}(x_0, y)}{C_0 T_1}\right)\left(|\hdel g_0|^2 +1\right)^{\frac{1}{2}}d\mu_h(y) \nonumber \\
        &= \int_0^{\infty} \frac{C_0}{T_1^{\frac{n}{2}}}\exp\left(-\frac{r^2}{C_0 T_1}\right)\left(\int_{\partial B(x_0, r)}\left(|\hdel g_0|^2 + 1\right)^{\frac{1}{2}}dS(y)\right)dr \nonumber \\
        &= \int_0^{\infty} \frac{C_0}{T_1^{\frac{n}{2}+1}}\exp\left(-\frac{r^2}{C_0 T_1}\right)\frac{2r}{C_0}\left(\int_{B(x_0, r)}\left(|\hdel g_0|^2 + 1\right)^{\frac{1}{2}}d\mu_h(y)\right)dr
    \end{align}
    where $dS$ is the surface area measure induced by $d\mu_h$ and $C_0$ above is the constant obtained from (\ref{eqn:gradient-estimate-K-heat-kernel-estimate}). For $p \geq 1$, assumption (ii) and H\"older's inequality gives
    \begin{equation*}
        \int_{B(x_0, r)} \left(|\hdel g_0|^2 + 1\right)^{\frac{1}{2}}d\mu_h(y) \leq CL_0r^{n-1+\frac{\delta}{p}}
    \end{equation*}
    for all $0<r<r_0$. If $r>r_0$, we estimate it using trivial bound from \eqref{eqn:gradient-estimate-maximal-T1} and volume comparison:
     \begin{equation*}
        \int_{B(x_0, r)} \left(|\hdel g_0|^2 + 1\right)^{\frac{1}{2}}d\mu_h(y) \leq  B_1t^{\frac\delta{2p}-\frac12}\cdot \exp\left(Cr\right).
    \end{equation*}

    Substituting this into (\ref{eqn:gradient-estimate-Q-heat-kernel-estimate2}) we obtain,
    \begin{align*}
        &\int_0^{\infty} \frac{C_0}{T_1^{\frac{n}{2}+1}}\exp\left(-\frac{r^2}{C_0 T_1}\right)\frac{2r}{C_0}\left(\int_{B(x_0, r)}\left(|\hdel g_0|^2 + 1\right)^{\frac{1}{2}}d\mu_h(y)\right)dr \nonumber \\
        &\quad \leq \int_0^{r_0} \frac{C}{T_1^{\frac{n}{2}+1}}\exp\left(-\frac{r^2}{C_0 T_1}\right)r^{n+\frac{\delta}{p}}dr + \int_{r_0}^{\infty} \frac{2B_1}{T_1^{\frac{n}{2}+\frac32-\frac{\delta}{p}}}\exp\left(Cr-\frac{r^2}{C_0 T_1}\right)  dr \nonumber \\
        &\quad \leq CT_1^{-\frac{1}{2}+\frac{\delta}{2p}}\int_0^\infty\exp\left(-\frac{s^2}{C_0}\right)s^{n+\frac{\delta}{p}}ds + \int_{\frac{1}{\sqrt{T_1}}}^\infty \frac{CB_1}{T_1^{1+\frac{n}{2}}}\exp\left(-\frac{r^2}{2C_0} \right) dr\\
        &\quad\leq CT_1^{-\frac{1}{2}+\frac{\delta}{2p}} +  \frac{CB_1}{T_1^{1+\frac{n}{2}}}\exp\left(-\frac{1}{2C_0T_1} \right).
    \end{align*}
  So we have
    \begin{align}\label{eqn:gradient-estimate-Q-heat-kernel-estimate4}
        \frac14 e^{-AT_1^{\frac{\delta}{p}}-AT_1}B_1T_1^{-\frac{1}{2}+\frac{\delta}{2p}} &\leq C_1 T_1^{-\frac{1}{2}+\frac{\delta}{2p}} +  \frac{CB_1}{T_1^{1+\frac{n}{2}}}\exp\left(-\frac{1}{2C_0T_1} \right)+ C T_1.
    \end{align}
    for some $C_1(n,\Lambda,k_1,k_2,L_0)>0$.
 By choosing $B_1=8C_1$, we see that $T_1$ is bounded from below by a constant depending only on $\delta,p,k_1,k_2,n,\Lambda,L_0$.
\end{proof}

\section{Short-time existence}\label{sec:proof-of-main-theorem}
In this section, we will prove short-time existence on metrics which are possibly singular and satisfy the Morrey-type integrability condition. We first consider the case when $\Sigma=\emptyset$, that is, when the initial data is smooth and satisfies the uniform Morrey-type integrability condition. We note that since $M$ is possibly non-compact, the uniform short-time existence is not covered by the work of Simon \cite{simon_deformation_2002}.
\begin{proof}[Proof of Existence in Theorem~\ref{thm:intro-application-1-statement} when $\Sigma=\emptyset$]

Without loss of generality, we assume $M$ to be non-compact. By \cite{tam_exhaustion_2010} and $|\text{Rm}(h)|\leq k_0$, there is $\rho \in C^\infty_\text{loc}(M)$ such that $|\hdel \rho|^2 + |\hdel^2 \rho| \leq 1$ and
    \begin{equation*}
        C(n,k_0)^{-1}(d_h(\cdot, p)+1)\leq \rho(\cdot) \leq  C(n,k_0)(d_h(\cdot,p) + 1)
    \end{equation*}
    for some $ C(n,k_0)>0$. Let $\phi$ be a smooth function on $[0,+\infty)$ such that $\phi \equiv 1$ on $[0,1]$, $\phi \equiv 0$ on $[2,+\infty)$ and $0\leq-\phi'\leq 10$. Fix $R_i\to+\infty$ and define a smooth metric 
    $$g_{i,0}=\phi(\rho/R_i)g_0+(1-\phi(\rho/R_i)) h$$
on $M$ which coincide with $h$ at spatial infinity of $M$ and coincide with $g_0$ on compact subset of $M$ as $i\to+\infty$. Then we might apply Theorem A.1 of \cite{lamm_ricci_2021} (which is a modification of Shi's classical existence theory in \cite{shi_deforming_1989}), to find  short-time solution to the Ricci-DeTurck $h$-flow $g_i(t)$ on $M \times [0,S_i]$ for some maximal existence time $S_i > 0$ such that $g_i(0) = g_{i,0}$ and $\sup\limits_M |\hdel^m g_i(\cdot, t)|<+\infty$ for all $t \in [0, S_i]$ and all $m\in \mathbb{N}$. 

Since 
\begin{equation}
|\hdel g_{i,0}|\leq \frac{C(n,k_0)}{R_i}|g_0-h|+\phi\cdot |\hdel g_0|,
\end{equation}
the metric $g_{i,0}$ satisfies assumption of  Proposition \ref{prop:gradient-estimate} uniformly for all $i\to+\infty$. Hence, Proposition \ref{prop:gradient-estimate} and Lemmas \ref{lem:estimates-bootstrap} apply to see that there is a $T(n, \Lambda_0, L_0, p, \delta,h) > 0$ such that $T \leq S_i$ with the solution $g_i(t)$ obtained above also satisfying (a) for all $t \in (0, T]$ uniformly for all $i\to+\infty$. We therefore obtain a smooth Ricci-Deturck $h$-flow $g(t)$ on $M\times (0,T]$ by taking subsequence $i\to+\infty$ using  the Arzel\`a--Ascoli theorem. Since $|\hdel g_i|$ has a uniform integrable bound in $t$, we see that $g(t)$ exists as a $C^0_{loc}$ metric on $M\times [0,T]$. Indeed, by shrinking $T$ we have 
\begin{equation}
(2\Lambda)^{-1}h\leq g(t)\leq 2\Lambda h
\end{equation}
on $M\times [0,T]$ by integrating $\partial_t g$ in time.  
\end{proof}

We now want to consider the case when $\Sigma$ is a non-empty compact subset of $M$. We start with constructing a $C^\infty$ approximation of $g_0$, see also \cite{lee_positive_2013,shi_scalar_2016,
lee_continuous_2021,grant_positive_2014} and references therein for the other similar approximation schemes.

\begin{lemma}\label{lem:mollification-scheme}
    Suppose $g_0$ is a $C^0_{loc}$ metric on $M$ and is smooth away from a compact set $\Sigma$. Moreover, suppose $g_0$ satisfies:
    \begin{enumerate}[(i)]
        \item there is a $\Lambda_0 > 1$ such that on $M$,
        \begin{equation*}
            \Lambda_0^{-1}h \leq g_0 \leq \Lambda_0 h;
        \end{equation*}
        \item there exist $L_0,\delta,r_0 > 0$ and $p\geq 1$ such that for every $x_0 \in M$, $0 < r < r_0$,
        \begin{equation*}
            \fint_{B(x_0,r)} |\hdel g_0|^p d\mu_h \leq L_0 r^{-p+\delta}.
        \end{equation*}
    \end{enumerate}
    Then there is a sequence of smooth metrics $g_{i,0}$ on $M$ such that 
    \begin{enumerate}[(a)]
        \item for $i$ sufficiently large,
        \begin{equation*}
            (2\Lambda_0)^{-1}h \leq g_{i,0} \leq 2\Lambda_0 h\quad\text{on}\;\;M;
        \end{equation*}
        \item  there exist $L_1,r_1 > 0$  such that for every $x_0 \in M$, $0 < r < r_1$,
        \begin{equation*}
            \fint_{B(x_0,r)} |\hdel g_{i,0}|^p d\mu_h \leq L_1 r^{-p+\delta};
        \end{equation*}
        \item $g_{i,0} \to g_0$ in $C^0_{loc}$ on $M$ as $i \to \infty$;
        \item $g_{i,0} \to g_0$ in $C^\infty_{loc}$ on $M\setminus\Sigma$ as $i \to \infty$;
        \item $g_{i,0} \to g_0$ in $W^{1,p}_\text{loc}(M)$ as $i\to+\infty$.
    \end{enumerate}
\end{lemma}
\begin{proof}
%

Without loss of generality, we will assume $M$ to be non-compact.  By the embedding theorem of Morrey (see for example ``Morrey's Lemma''  \cite[(1.3) ]{adams_morrey_2015}), assumption (ii) implies that $g_0$ is locally H\"older continuous with exponent $\delta/p < 1$ on $M$.  Since $\Sigma$ is compact, there is a point $p \in M$ and $R > 0$ such that $\Sigma \Subset B(p,R)$.

Note that $K=\overline{B(p,R)}$ is compact. We now cover $K$ by finitely many open coordinate charts $\{U_k\}_{k=1}^N$. Also let $U_0 = M \setminus K$. Let $\varphi_k$ be a partition of unity subordinate to $U_0 \cup \bigcup\limits_{k=1}^N U_k$. Then we can decompose the metric $g_0=\sum_{k=0}^N g^k_0$ on each chart by $g_0^k = \varphi_k g_0$. 
We might assume each $U_k$ is diffeomorphic to unit ball $B_{\mathbb{R}^n}(1)$ in $\mathbb{R}^n$. For $k = 1,\dots,N$, let $\eta$ be the standard mollifier with compact support inside $B_{\mathbb{R}^n}(1)$ and $\int_{U_k} \eta(y)dy = 1$. For $k\geq 1$, we define
    \begin{equation*}
        g_{i,0}^k(x) := \int_{U_k} g_0^k(x-i^{-1}y)\eta(y)dy.
    \end{equation*}
    for $x\in U_k$. Since $\varphi_k$ is compactly supported in $U_k$ for $k\geq 1$, we see that $g^k_{i,0}(x)=0$ for $x\to \partial U_k$ and sufficiently large $i$. Hence, $g^k_{i,0}$ extends trivially on $M$. Thus we define
$
        g_{i,0}:= \sum\limits_{k=1}^N g_{i,0}^k + g_0^0$ to be a metric on $M$.
  Near the infinity of $M$, $g_{i,0}=g_0^0=g_0$, while on compact set, $g_{i,0}$ is a mollification of $g_0$. It suffices to check that the above properties are satisfied on compact sets. For (a) and (b) above, it suffices to show that they are satisfied on compact set. Clearly by the standard mollification, we have (e) and (d) since $g_0\in W^{1,p}_{loc}(M)\cap C^\infty_{loc}(M\setminus \Sigma)$ so that $g_{i,0}^k\to g_0^k$ as $i\to+\infty$ in  $W^{1,p}_{loc}(M)\cap C^\infty_{loc}(M\setminus \Sigma)$ for each $k$.
    
To see (c), If $x\notin \cup_{k=1}^N U_k$, then the result holds directly.  If $x \in K$, then $x \in U_k$ for some $k = 1,\dots,N$. Fix any $\e>0$. By the H\"older continuity of $g_0$, for $i$ sufficiently large, $|g_0^k(x-i^{-1}y) - g_0^k(x)| < C|i^{-1}y|^\alpha <\varepsilon$ for any $x, y \in U_k$. Here $C$ depends also on the choice of partition of unity. Hence by $\int_{U_k} \eta(y)dy = 1$,
    \begin{align*}
        |g_{i,0}^k(x) - g_0^k(x)| &= \left|\int_{U_k} g_0^k(x-i^{-1}y)\eta(y)dy - g_0^k(x)\right| \nonumber \\
        &\leq \int_{U_k} | g_0^k(x-i^{-1}y) - g_0^k(x)|\eta(y)dy \nonumber \\
        &\leq \int_{U_k} |i^{-1}y|^\alpha \eta(y)dy \nonumber \\
        &\leq \int_{U_k} \varepsilon \eta(y)dy \nonumber = \varepsilon.
    \end{align*}
    From this, (a) also follows  for $i$ sufficiently large.
    

  It remains to prove (b). It suffices to show (b) for each $g_{i,0}^k$. We choose $r_1<r_0$ so that for each $x_0\in \mathrm{supp}(\varphi_k)$, $B(x_0,r_1)\Subset U_k$. For fixed $y$ with $|y| \leq 1$, note that $z:= x-i^{-1}y$ is a translation with determinant of Jacobian uniformly bounded from above and below. We use the fact that mollification and differentiation commute, i.e. $\partial g_i(x) = (\eta_i \ast \partial g)(x)$, then by Minkowski's integral inequality,
    \begin{align*}
        &\left(\int_{B(x_0,r)}|\partial g_{i,0}^k(x)|^pd\mu_h(x)\right)^\frac{1}{p} \nonumber \\
        &\quad= \left(\int_{B(x_0,r)}\left|\int_{U_k} |\partial g_0^k(x-i^{-1}y)\eta(y)dy\right|^{p}d\mu_h(x)\right)^\frac{1}{p} \nonumber \\
        &\quad\leq \int_{U_k} \left(\int_{B(x_0,r)}|\partial g_0^k(x-i^{-1}y)\eta(y)|^pdx\right)^\frac{1}{p}dy \nonumber \\
        &\quad\leq \int_{U_k}\eta(y)dy\left(\int_{B(x_0,r)} |\partial g_0^k(x-i^{-1}y)|^p d\mu_h(x)\right)^\frac{1}{p} \nonumber \\
        &\quad\leq \left(C\int_{B(x_0 + i^{-1}y, r)} |\partial g_0^k(z)|^p dz\right)^\frac{1}{p} \nonumber \\
        &\quad\leq \left(CL_0 r^{n-p+\delta}\right)^\frac{1}{p}.
    \end{align*}
 Here $C$ depends on the choice of partition of unity. So (b) is satisfied.
    
%
%
\end{proof}

Now we are ready to prove Theorem~\ref{thm:intro-application-1-statement} under the presence of singular sets $\Sigma$.
\begin{proof}[Proof of Theorem~\ref{thm:intro-application-1-statement} when $\Sigma\neq \emptyset$ ]

Let $g_{i,0}$ be the smooth approximation of $g_0$ obtained from Lemma~\ref{lem:mollification-scheme}.   By properties (a), (b) in Lemma \ref{lem:mollification-scheme} above, we know that for $i$ sufficiently large, that
    \begin{equation}\label{eqn:rough-initial-gi0-Linfty}
        (2\Lambda_0)^{-1}h \leq g_{i,0} \leq 2\Lambda_0 h,
    \end{equation}
    and for all $x_0\in M$ and $0<r<r_1$,
    \begin{equation}\label{eqn:rough-initial-gi0-Morrey}
        \fint_{B(x_0, r)} |\hdel g_{i,0}|^p d\mu_h \leq L r^{-p+\delta}
    \end{equation}
    for some $L> 0$. 
    
    Since each $g_{i,0}$ is smooth and satisfies (\ref{eqn:rough-initial-gi0-Linfty}) and (\ref{eqn:rough-initial-gi0-Morrey}) above for $i$ large enough, by Theorem 1.1, we obtain constants $T(n,h,\Lambda_0,p,\delta,h,L,r_1) > 0$ and a sequence of $B_k(n,h,\Lambda_0,p,\delta,h,L,r_1) > 0$ and solutions $g_i(t)$ to the Ricci-DeTurck $h$-flow on $M \times [0,T]$ such that $g_i(0) = g_{i,0}$ satisfying
$$\sup\limits_{M}|\hdel^k g_i(t)| \leq \frac{B_k}{t^{\frac{1}{2}(m-\frac{\delta}{p})}}$$ on $M \times (0,T]$.


So by the Arzel\`a--Ascoli theorem and taking a subsequence, we obtain smooth $g(t) = \lim\limits_{i\to\infty}g_i(t)$ on $M \times (0,S]$ with $g(t)$ satisfying (a) and (b). Since $g_{i,0}\to g_0$ in $C^\infty_{loc}(M\setminus \Sigma)$ as $i\to+\infty$, $g(t)\in C^\infty_{loc}\left((M\setminus \Sigma)\times [0,T]\right)$ using \cite[Proposition 2.2]{chu_ricci-deturck_2022}. 

It remains to show the uniqueness within the class of solution satisfying the property (a). If $M$ is compact, it follows from \cite{burkhardt-guim_pointwise_2019} since $g(t)$ attains $g_0$ in $C^0_{loc}$ sense.
It remains to consider the case when $M$ is non-compact. It follows almost identically to the proof of \cite[Theorem 3.1]{chu_ricci-deturck_2022} with modification. We give a sketch here. Suppose $g(t)$ and $\hat g(t)$ are two solution on $M\times (0,T]$ satisfying derivatives estimates in (a). Clearly, both $g(t)$ and $\hat g(t)$ converge to $g_0$ in $C^0_{loc}$ as $t\to 0$. Instead of using the concentration of gradient, we use the fact that $|\hdel g|+|\hdel \hat g|\leq Ct^{-\frac12 +\frac{\delta}{2p}}$ to show that the tensor $\a=g-\hat g$ satisfies 
\begin{equation}
\frac{\partial}{\partial t}\int_M \eta^4 |\a|^2 \,d\mu_h \leq C\int_M \eta^2 |\a|^2 \,d\mu_h +Ct^{-\frac12 +\frac{\delta}{2p}}\int_M \eta^4 |\a|^2 \,d\mu_h
\end{equation}
which  is analogous to \cite[(3.20)]{chu_ricci-deturck_2022}. Here $\eta$ is a smooth cut-off function on $M$ with compact support on $B(x,2)$. Since $t^{-\frac12 +\frac{\delta}{2p}}$ is integrable in $t$ near $t=0$, the argument can be carried over.
\end{proof}

\section{Applications: preserving distributional scalar curvature lower bound}\label{sec:applications}

In this section, we use the existence theory established above to study the preserving of weak scalar curvature lower bound along the Ricci flow smoothing.
In \cite{lee_positive_2015}, Lee--LeFloch introduced a notion of distributional scalar curvature that is defined for metrics with low regularity. In particular, for any $g \in L^\infty_{\text{loc}}\cap W^{1,2}_{\text{loc}}$ with locally bounded inverse $g^{-1}\in L^\infty_{\text{loc}}$, they define the scalar curvature distribution $R_g$ by
\begin{equation}\label{eqn:distributional-scalar-defn}
    \left\langle R_g, u\right\rangle := \int_M \left(-V\cdot\hdel\left(u\frac{d\mu_g}{d\mu_h}\right)+Fu\frac{d\mu_g}{d\mu_h}\right)d\mu_h
\end{equation}
for every compactly supported smooth test function $u:M\to\mathbb{R}$ and where
\begin{align*}
    &\Psi^k_{ij} := \frac{1}{2}g^{k\ell}\left(\hdel_ig_{j\ell}+\hdel_jg_{i\ell}-\hdel_{\ell}g_{ij}\right) \\
    &V^k := g^{ij}\Psi^{k}_{ij} - g^{ik}\Psi^{j}_{ji} \\
    &F := \tr_g\widetilde{Ric} - \hdel_kg^{ij}\Psi^k_{ij} + \hdel_kg^{ik}\Psi^{j}_{ji} + g^{ij}\left(\Psi^{k}_{k\ell}\Psi^{\ell}_{ij} - \Psi^{k}_{j\ell}\Psi^{\ell}_{ik}\right)
\end{align*}
and $\frac{d\mu_g}{d\mu_h}\in L^\infty_{\text{loc}}\cap W^{1,2}_{\text{loc}}$ is the density of $d\mu_g$ with respect to $d\mu_h$. Let $a$ be a constant. We say $R_g \geq a$ in the distributional sense when $    \left\langle R_g-a, u\right\rangle\geq 0$ for every non-negative smooth test function $u$. Clearly, when $g$ is $C^2$, the distributional scalar curvature $R_g$ coincides with the classical scalar curvature.

Our main motivation is to study metrics on compact manifolds which are smooth outside some singular sets. In other words, we want to consider the question of whether a scalar curvature lower bound can be extended to a distributional scalar curvature lower bound across a set where the metric is singular.  In \cite{jiang_removable_2022}, the authors consider this for when $g \in C^0 \cap W^{1,p}_{loc}(M)$ for $n \leq p \leq \infty$ and the singular set is small in the sense of Hausdorff dimension, see  \cite[Lemma 2.7]{jiang_removable_2022}. In this work, we use a notion of co-dimension for compact subsets $\Sigma \subset M$ based on the volume growth of tubular neighborhoods of $\Sigma$. 

\begin{definition}
\label{defn:codim}
For a complete smooth Riemannian manifold $M^n$ with smooth background metric $h$, a compact set $\Sigma$ of $M$ is said to have co-dimension at least $d_0 > 0$ if there exist $b > 0$ and $C > 0$ such that for all $0 < \varepsilon \leq b$
\begin{equation*}
    \text{Vol}_h\left(\Sigma(\varepsilon)\right) = \text{Vol}_h\left(\left\{x \in M : d_h(x,\Sigma) < \varepsilon\right\}\right) \leq C\varepsilon^{d_0}.
\end{equation*}
    
\end{definition}

We show that when $g$ satisfies the $L^\infty$ and the Morrey-type condition, and satisfies a scalar curvature lower bound outside of a compact set $\Sigma$, then the corresponding distributional scalar curvature lower bound holds when $\Sigma$ is not too large in the sense of definition~\ref{defn:codim}.

\begin{proposition}
\label{prop:distributional-scalar-curvature}
    Let $M^n$ be an $n$-dimensional manifold with smooth background metric $h$ satisfying \eqref{eqn:h-remark-curvature-estimates}. Suppose the metric $g$ satisifes:
    \begin{enumerate}[(i)]
        \item there is a $\Lambda > 0$ such that $\Lambda^{-1}h \leq g \leq \Lambda h$ on $M$;
        \item for $p \geq 2, \delta > 0, r_0 > 0$ there is a $L$ such that for all $x_0 \in M$ and $0 < r < r_0$,
        \begin{equation}
            \fint_{B(x_0, r)} |\hdel g|^p d\mu_h \leq L r^{-p + \delta};
        \end{equation}
        \item there is a compact $\Sigma \subset M$ with co-dimension $d$ at least $2 - \frac{\delta'}{p}$ for some $0<\delta'<\delta$ such that $g$ is smooth on $M \setminus \Sigma$ and $R_g \geq a$ holds in the classical sense on $M \setminus \Sigma$.
    \end{enumerate}
    Then for all smooth compactly supported non-negative test functions $u$, the distributional scalar curvature satisfies $\left\langle R_g - a, u \right\rangle \geq 0$.

\end{proposition}

\begin{proof}
Fix a smooth compactly supported non-negative  function $u$. 
We will use $C$ to denote any constants depending only on $n,\Lambda,p,\delta,r_0,L>0$ and $u$ which might be adjusted from line to line.  As in \cite{jiang_removable_2022}, for any $\varepsilon > 0$, we let $\eta_\varepsilon$ be a smooth non-negative function so that $\eta_\varepsilon \equiv 1$ on $\Sigma(\varepsilon)$ and $\eta_\varepsilon \equiv 0$ on $M\setminus\Sigma$ with $|\hdel \eta_\varepsilon|\leq C\varepsilon^{-1}$. Then
    \begin{equation*}
        \langle R_{g} - a, u\rangle = \langle R_g - a, \eta_\varepsilon u\rangle + \langle R_g - a, (1-\eta_\varepsilon)u\rangle.
    \end{equation*}
    Since the support of $(1-\eta_\varepsilon)u$ is outside $\Sigma$, we have
    \begin{equation*}
        \langle R_g - a, (1-\eta_\varepsilon)u\rangle = \int_{M\setminus \Sigma(\varepsilon)} (R_g - a)(1-\eta_\varepsilon)u d\mu_g \geq 0
    \end{equation*}
    because $g$ is smooth and satisfies $R_g \geq a$ in the classical sense outside of $\Sigma(\varepsilon)$. So it suffices to show
    \begin{equation*}
        \lim\limits_{\varepsilon \to 0}\left|\langle R_g - a,\eta_\varepsilon u\rangle\right| = 0.
    \end{equation*}
    Then, again as in Lemma 2.7 of \cite{jiang_removable_2022}, we have by definition of $V, F$ and H\"older inequality,
    \begin{align}\label{eqn:distributional-scalar-curvature-integral-estimate}
        \left|\langle R_g - a,\eta_\varepsilon u\rangle\right| & \leq \int_M |V| \cdot \left| \hdel \left(\eta_\varepsilon u \frac{d\mu_g}{d\mu_h}\right)\right| d\mu_h + \int_M |F - a|\cdot \eta_\varepsilon u \frac{d\mu_g}{d\mu_h}d\mu_h \nonumber \\
        &\leq C\left(\int_{\Sigma(\varepsilon)}|\hdel g|^pd\mu_h\right)^\frac{1}{p}\text{Vol}_h(\Sigma(\varepsilon))^{1-\frac{1}{p}} \nonumber \\
        &\quad + C\left(\int_{\Sigma(\varepsilon)}|\hdel g|^pd\mu_h\right)^\frac{1}{p}\left(\int_{\Sigma(\varepsilon)}|\hdel\eta_\varepsilon|^\frac{p}{p-1}d\mu_h\right)^\frac{p-1}{p} \nonumber \\
        &\quad + C\left(\int_{\Sigma(\varepsilon)}|\hdel g|^pd\mu_h\right)^\frac{2}{p}\text{Vol}_h(\Sigma(\varepsilon))^{1-\frac{2}{p}} + C\text{Vol}_h(\Sigma(\varepsilon)) \nonumber \\
        &=: \text{I} + \text{II} + \text{III} + \text{IV}
    \end{align}
    where the constants above depend on $n,\Lambda$ and the $C^1$ norm of $u$.
    
    Since $\Sigma$ has co-dimension at least $d$, by definition there is a $C > 0$ such that
    \begin{equation*}
        \text{Vol}_h(\Sigma(\varepsilon)) \leq C\varepsilon^d
    \end{equation*}
    for  all $\varepsilon$ sufficiently small. Since $\Sigma$ is compact and $M$ carries the structure of a metric space with distance function $d_h$, $\Sigma$ is totally bounded, that is, for every fixed radius, it can be covered by a finite number of balls of that radius measured with respect to $d_h$. In particular, it can be covered by a finite number of balls of radius $\varepsilon/2$. Let $N$ be the minimal such number of balls, $B(x_k, \frac{\varepsilon}{2}), k = 1,\dots,N$. We first claim that
    \begin{equation*}
        \Sigma \subset \bigcup\limits_{k=1}^N B\left(x_k,\frac{\varepsilon}{2}\right) \subset \Sigma(\varepsilon).
    \end{equation*}
    The first inclusion is obvious. Now suppose $x \in B(x_k,\frac{\varepsilon}{2})$ for some $k$. By minimality of $N$, $B(x_k,\frac{\varepsilon}{2})\cap\Sigma \neq \emptyset$ and so $d_h(\Sigma,x_k) < \frac{\varepsilon}{2}$. Then by the triangle inequality we have
    \begin{equation*}
        d_h(\Sigma, x) \leq d_h(\Sigma, x_k) + d_h(x_k, x) < \frac{\varepsilon}{2} + \frac{\varepsilon}{2} = \varepsilon.
    \end{equation*}
    So the second inequality holds. Hence, we have
    \begin{equation*}
        \sum_{k=1}^N\text{Vol}_h\left(B\left(x_k,\frac{\varepsilon}{2}\right)\right) \leq \text{Vol}_h(\Sigma(\varepsilon)) \leq C\varepsilon^d.
    \end{equation*}
    By taking $\varepsilon$ small enough, we can also assume there is a constant $D=D(n,h)$ such that
    \begin{equation*}
        D^{-1}\varepsilon^n \leq \text{Vol}_h(B(\varepsilon)) \leq D\varepsilon^n.
    \end{equation*}
    So we have
    \begin{equation}\label{eqn:distributional-scalar-curvature-N-estimate}
        N \leq C\varepsilon^{d-n}.
    \end{equation}
    Similarly by the triangle inequality, we have that
    \begin{equation*}
        \Sigma(\varepsilon) \subset \bigcup\limits_{k=1}^N B\left(x_k,\frac{3\varepsilon}{2}\right)
    \end{equation*}
    since if $x \in \Sigma(\varepsilon)$, then there is a $y \in \Sigma$ such that $d_h(y,x) < \varepsilon$, and by $\Sigma \subset \bigcup\limits_{k=1}^N B\left(x_k,\frac{\varepsilon}{2}\right)$, there is an $x_k$ such that $d_h(x_k, y) < \frac{\varepsilon}{2}$, so $d_h(x,x_k) \leq d_h(x,y) + d_h(y,x_k) < \frac{3\varepsilon}{2}$.

Clearly, we have from definition that
 \begin{equation*}
        \text{IV} \leq C\varepsilon^d.
    \end{equation*}
    With (\ref{eqn:distributional-scalar-curvature-N-estimate}) and the Morrey assumption (ii), we can estimate the terms on the right hand side of (\ref{eqn:distributional-scalar-curvature-integral-estimate}) above. For I, we have
    \begin{align*}
        \text{I} &\leq C\left(\sum\limits_{k=1}^N \int_{B(x_k,\frac{3\varepsilon}{2})}|\hdel g|^pd\mu_h\right)^\frac{1}{p}\text{Vol}_h(\Sigma(\varepsilon))^{1-\frac{1}{p}} \nonumber \\
        &\leq C\left(NL\varepsilon^{n-p+\delta}\right)^\frac{1}{p}\text{Vol}_h(\Sigma(\varepsilon))^{1-\frac{1}{p}} \nonumber \\
        &\leq C(\varepsilon^{d-n+n-p+\delta})^\frac{1}{p}\varepsilon^{d\left(1-\frac{1}{p}\right)} \nonumber \\
        &= C\varepsilon^{d-1+\frac{\delta}{p}}.
    \end{align*} 
Argue in the exact same way, we have 
    \begin{align*}
        \text{III} &\leq C\varepsilon^{d-2+\frac{2\delta}{p}}.
    \end{align*}
    
%

It remains to estimate $\text{II}$. Using $|\hdel \eta|\leq C\e^{-1}$, we have 
\begin{equation}
\text{II}\leq C\e^{d-2+\frac{\delta}{p}}
\end{equation}
so that $\text{II}\to0$ as $\e\to 0$ since $d\geq 2-\frac{\delta'}{p}$ for some $\delta'<\delta$.    Substituting these back into (\ref{eqn:distributional-scalar-curvature-integral-estimate}) and by the value of $d$, we get that the right hand side of (\ref{eqn:distributional-scalar-curvature-integral-estimate}) vanishes as $\varepsilon\to0^+$, and so we are done.
\end{proof}

\subsection{Preservation of distributional  lower bound}

We now consider the preservation of distributional scalar curvature lower bounds along the Ricci flow. In \cite{jiang_weak_2021}, the authors showed that scalar curvature lower bounds in the distributional sense as above are preserved along the Ricci flow for initial metrics $g \in W^{1,p}(M^n)$ for $3 \leq n < p \leq \infty$. In this section, we improve the result by relaxing the regularity to Morrey-type condition with $2 \leq p \leq n$. We follow the approach of \cite{jiang_weak_2021} with modifications.

We first have the following lemma.

\begin{lemma}\label{lem:hessiansquared-spacetime-integrability}
    Let $(M^n,h)$ be a closed manifold and $g_0$ is a $C^0$ metric on $M$ satisfying. Let $g_0$ and $g(t)$ be as in Theorem \ref{thm:intro-application-1-statement}. Then there is a constant $C = C(n, M,h, \Lambda, L_0, p, \delta,r_0) < +\infty$ such that we have
    \begin{equation*}
        \int_0^T \int_M |\hdel^2 g(t)|^2 d\mu_h dt \leq C.
    \end{equation*}
\end{lemma}

\begin{proof}
    Recall that from (\ref{eqn:prelim-ricci-deturck-flow}), standard computations and the Cauchy-Schwarz inequality yields
    \begin{equation*}
        \frac{\partial}{\partial t} |\hdel g|^2 \leq g^{ij}\hdel_i\hdel_j |\hdel g|^2 - C_1 |\hdel^2 g|^2 + C_2|\hdel g|^4 + C_3
    \end{equation*}
    for constants $C_1, C_2, C_3$ that are only depending on $n, h$ where we have gradient estimate:
    \begin{equation*}
        C_2|\hdel g|^4 \leq C_2B_1|\hdel g|^2t^{-1+\frac{\delta}{p}},
    \end{equation*}
    so defining $f(x,t) = |\hdel g|^2\exp\left(-At^\frac{\delta}{p}\right)$ for some constant $A > 0$ to be determined, we have
    \begin{align*}
        \frac{\partial}{\partial t} f(t) &\leq \exp\left(-At^\frac{\delta}{p}\right)g^{ij}\hdel_i\hdel_j|\hdel g|^2 - C_1|\hdel^2 g|^2 \nonumber \\
        &\quad + C_2f(t)t^{-1+\frac{\delta}{p}} - Af(t)t^{-1+\frac{\delta}{p}} + C_3.
    \end{align*}
    where we are using the fact that for $t \in [0, T]$ we have that $\exp\left(-At^\frac{\delta}{p}\right)$ is uniformly bounded from above. Integrating by parts and Young's inequality yields
    \begin{equation*}
        \frac{\partial}{\partial t} \int_M fd\mu_h \leq \int_M -C_1|\hdel^2g|^2d\mu_h + \int_M\left(C_2+C_4-A\right)f(t)t^{-1+\frac{\delta}{p}}d\mu_h + C_3.
    \end{equation*}
    Now we choose $A$ large enough so that it dominates $C_2 + C_4$ and we obtain
    \begin{equation*}
        \frac{\partial}{\partial t} \int_M fd\mu_h + \int_M C_1|\hdel^2 g|^2d\mu_h \leq C_3.
    \end{equation*}
    
Since $M$ is compact, by covering argument we see that 
\begin{equation}
\int_M |\tilde\nabla g_0|^2 d\mu_h \leq C(n,M,h,r_0,L_0,p,\delta).
\end{equation} 
With this, we integrate in time from from $0$ to $T$ to obtain the result.
\end{proof}

Now let $M$ be a closed manifold and let $g(t)$ be the Ricci-DeTurck $h$-flow on $M$. 
Consider the heat kernel:
\begin{equation*}
    \left(\partial_t - \Delta_x-\nabla_X\right)K(\cdot,\cdot;y,s) = 0, \quad \lim\limits_{t\to s^+}K(\cdot,t;y,s) = \delta_y(\cdot)
\end{equation*}
and moreover satisfies for the conjugate heat operator:
\begin{equation*}
    \left(-\partial_s - \Delta_y - \del_X + R\right)K(x,t;\cdot,\cdot) = 0, \quad \lim\limits_{s\to t^-}K(x,t;\cdot,s) = \delta_x(\cdot).
\end{equation*}
where the Laplacian operators are taken with respect to $g(t)$. Let $\tilde{u}$ be an arbitrary non-negative $C^\infty$ function on $M$. We consider the conjugate heat equation with $\tilde{u}$ as final data, that is,
\begin{equation}\label{eqn:conjugate-heat-eqn}
    \begin{cases}
        &\frac{\partial}{\partial t} u = -\Delta u - \del_X u + Ru \quad \text{ on } M \times[0,T] \\
        &u\big|_{t=T} = \tilde{u}
    \end{cases}
\end{equation}
where $R$ is the scalar curvature with respect to $g(t)$. By the maximum principle, the solution is non-negative and given by convolving with the fundamental solution:
\begin{equation*}
    u(x,t) = \int_M K(y,T;x,t)\tilde{u}(y)d\mu_{g(T)}(y).
\end{equation*}

Similar to Proposition 4.1 of \cite{jiang_weak_2021}, we have the estimates for $u$ in the following lemma.
\begin{lemma}\label{lem:conjugate-heat-equation-estimates}
    Let $(M,h)$ be a closed Riemannian manifold and $g_0$ is a $C^0_{loc}\cap W^{1,p}_{loc}$ metric on $M$ satisfying assumptions in Theorem \ref{thm:intro-application-1-statement} with $p\geq 2$. Suppose $g(t)$ is the unique solution obtained from Theorem \ref{thm:intro-application-1-statement} and $u$ is a solution obtained as above. Then there exists $C_0>0$ depending only on $n,M,h,\Lambda,L_0,p,\delta,\tilde u$ such that
    \begin{enumerate}[(a)]
        \item $u(\cdot,t) \leq C_0$, for all $t \in [0,T];$
        \item $\displaystyle\int_M |\del_{g(t)} u(\cdot,t)|^2 d\mu_{g(t)} \leq C_0$, for all $t \in [0,T];$
        \item For any $a\in \mathbb{R}$, $\displaystyle\int_M (R_{g(t)} - a)u d\mu_{g(t)}$ is monotonically non-decreasing with respect to $t$.
    \end{enumerate}
\end{lemma}
\begin{proof}
The proof is identical to that of \cite[Proposition 4.1]{jiang_weak_2021} using Lemma~\ref{lem:hessiansquared-spacetime-integrability} and the integrable interior estimates from Theorem~\ref{thm:intro-application-1-statement}.
\end{proof}

With these two lemmata, we have the following main application for the compact case. 

\begin{theorem}\label{thm:preservation-distributional-scalar-curvature-lower-bound}
    Let $(M^n,h)$ be a  closed manifold and $g_0$ be a $C^0$ metric satisfying
    \begin{enumerate}[(i)]
        \item there is a $\Lambda > 0$ such that $\Lambda^{-1}h \leq g_0 \leq \Lambda h$ on $M$;
        \item there exist $p \geq 2, \delta > 0, r_0 ,L> 0$ such that for all $x_0 \in M$ and $0 < r < r_0$,
        \begin{equation}
            \fint_{B(x_0, r)} |\hdel g_0|^p d\mu_h \leq L r^{-p + \delta};
        \end{equation}
   \item $R_{g_0}\geq a$ in the distributional sense on $M$ for some $a\in \mathbb{R}$.
    \end{enumerate}
    Let $g(t), t \in (0, T]$ be the Ricci-DeTurck $h$-flow from $g_0$ obtained in Theorem \ref{thm:intro-application-1-statement} with $2 \leq p \leq n$. Then for any $t \in (0, T]$, $R_{g(t)} \geq a$ in the classical sense on $M$.
\end{theorem}

\begin{proof}
Fix any $t>0$. It suffices to show 
    \begin{equation*}
        \langle   R_{g(t)}-a, \tilde u \rangle  \geq 0
    \end{equation*}
    for any arbitrary nonnegative function $\tilde{u}\in C^\infty(M)$ where $\langle R,u\rangle $ is defined as in \eqref{eqn:distributional-scalar-defn}. To do this, we follow the strategy as in Theorem 5.1 of \cite{jiang_weak_2021}, but modify it for our setting. 
    
We pay extra attention to the construction of $g(t)$.  Let $g_{i,0}$ be as in Lemma \ref{lem:mollification-scheme} and $g_i(t)$ be solutions to the Ricci-DeTurck $h$-flow starting from $g_{i,0}$ as in Theorem \ref{thm:intro-application-1-statement} above so that $g(t)$ is the smooth limit of $g_i(t)$ as $i\to+\infty$ on $M\times (0,T]$.

Let $u_i$ be the solution to (\ref{eqn:conjugate-heat-eqn}) (with respect to $g_i(t)$) with $u_{i,t} = \tilde{u}$ and $u_{i,0}(x) = u_i(x,0)$. Note that $u_i$ satisfies the conclusion in Lemma \ref{lem:conjugate-heat-equation-estimates} above uniformly as $i\to+\infty$. In particular, we have
    \begin{equation}\label{eqn:scalar-monotonicity}
        \int_M (R_{g_i(t)}-a)\tilde{u}d\mu_{g_i(t)} \geq \int_M (R_{g_{i,0}}-a)u_{i,0}d\mu_{g_{i,0}}.
    \end{equation}
    Since by hypothesis we have that $\langle R_{g_0} - a,u_{i,0}\rangle \geq 0$, what remains is to examine the difference
    \begin{equation*}
        \left|\langle  R_{g_{i,0}}-a, u_{i,0}\rangle  - \langle R_{g_0} - a, u_{i,0}\rangle \right|.
    \end{equation*}

    By the triangle inequality,
    \begin{align*}
        \left|\left\langle R_{g_{i,0}}-a, u_{i,0}\right\rangle - \left\langle R_{g_0} - a, u_{i,0}\right\rangle \right| &\leq \left|\langle R_{g_{i,0}},u_{i,0}\rangle - \langle R_{g_0}, u_{i,0}\rangle\right| \nonumber \\
        &\quad + |a|\left|\int_M u_{i,0} d\mu_{g_{i,0}} - \int_M u_{i,0} d\mu_{g_0}\right|
    \end{align*}
    Since $g_{i,0} \to g_0$ globally in $C^0$, we observe that 
    \begin{equation*}
        \lim\limits_{i\to\infty}\left\lVert \frac{d\mu_{g_{i,0}}}{d\mu_{g_0}} - 1\right\rVert_{C^0(M)} = 0
    \end{equation*}
and
    \begin{equation*}
        \lim\limits_{i\to\infty}\left\lVert \frac{d\mu_{g_0}}{d\mu_h} - \frac{d\mu_{g_{i,0}}}{d\mu_h}\right\rVert_{C^0(M)} = 0.
    \end{equation*}

    Then to handle the second term above, by (a) of Lemma \ref{lem:conjugate-heat-equation-estimates} above, we have
    \begin{align*}
        |a|\left|\int_M u_0 d\mu_{g_{i,0}} - \int_M u_0 d\mu_{g_0}\right| &\leq |a|\left|\int_M u_0\left(\frac{d\mu_{g_{i,0}}}{d\mu_{g_0}} - 1\right)d\mu_{g_0}\right| \nonumber \\
        &\leq C\left\lVert \frac{d\mu_{g_{i,0}}}{d\mu_{g_0}} - 1\right\rVert_{C^0(M)}
    \end{align*}
    for some positive constant $C$ depending on $a, n, k_1, k_2, \Lambda, p, \delta, \tilde{u}$.

    For the first term, we expand out the definition of distributional scalar curvature. Let $V, F$ be the vector and scalar fields associated with $g_0$, and let $V_i, F_i$ be the vector and scalar fields associated with $g_{i,0}$ that appear in the definition for distributional scalar curvature (\ref{eqn:distributional-scalar-defn}). Then by triangle inequality, we have
    \begin{align*}
        &\left|\langle R_{g_{i,0}},u_{i,0}\rangle - \langle R_{g_0}, u_{i,0}\rangle\right| \nonumber \\
        &\quad \leq \int_M |V - V_i|\left|\hdel\left(u_{i,0}\frac{d\mu_{g_{i,0}}}{d\mu_h}\right)\right|d\mu_h + \int_M |V| \left|\hdel\left(u_{i,0}\frac{d\mu_{g_0}}{d\mu_h} - u_{i,0}\frac{d\mu_{g_{i,0}}}{d\mu_h}\right)\right|d\mu_h \nonumber \\
        &\qquad + \int_M |F_i u_{i,0} - F u_{i,0}|\left|\frac{d\mu_{g_{i,0}}}{d\mu_h}\right|d\mu_h + \int_M |Fu_{i,0}|\left|\frac{d\mu_{g_{i,0}}}{d\mu_h} - \frac{d\mu_{g_0}}{d\mu_h}\right|d\mu_h \nonumber \\
        &\quad =: \text{I} + \text{II} + \text{III} + \text{IV}
    \end{align*}
    and we estimate each of the above terms separately. In the following, the constants $C_m, m \in \mathbb{N}$ below will come from (a) and (b) of Lemma \ref{lem:conjugate-heat-equation-estimates} and the Morrey condition on $g_0$. In other words, $C_m$ will at most depend on $n, k_1, k_2, \Lambda, p, \delta, \tilde{u}, L, h$. The constants may be changing line by line.

    Before handling each term, observe that since $M$ is compact, we have that $g_{i,0} \to g_0$ in $W^{1,p}(M)$ by (e) of Lemma \ref{lem:mollification-scheme} above. 
    
%

    For I, we have by H\"older inequality,
    \begin{align*}
        \text{I} &= \int_M |V - V_i||\hdel u_{i,0}|\left|\frac{d\mu_{g_{i,0}}}{d\mu_h}\right|d\mu_h + \int_M |V - V_i||u_{i,0}|\left|\hdel\frac{d\mu_{g_{i,0}}}{d\mu_h}\right|d\mu_h \nonumber \\
        &\leq C_1\int_M |\hdel g_0 - \hdel g_{i,0}||\hdel u_{i,0}|d\mu_h + C_2\int_M |\hdel g_0 - \hdel g_{i,0}||\hdel g_{i,0}|d\mu_h \nonumber \\
        &\leq C_1\left(\int_M |\hdel g_{i,0} - \hdel g_0|^pd\mu_h\right)^\frac{1}{p}\left(\int_M |\hdel u_{i,0}|^2d\mu_h\right)^\frac{1}{2}\text{Vol}_h(M)^\frac{p-2}{2p} \nonumber \\
        &\quad + C_2\left(\int_M |\hdel g_{i,0} - \hdel g_0|^pd\mu_h\right)^\frac{1}{p}\left(\int_M |\hdel g_{i,0}|^pd\mu_h\right)^\frac{1}{p}\text{Vol}_h(M)^\frac{p-2}{p} \nonumber \\
        &\leq C_1\left\lVert g_{i,0} - g_0\right\rVert_{W^{1,p}(M)} + C_2\left\lVert g_{i,0} - g_0\right\rVert_{W^{1,p}(M)}
    \end{align*}
    where for the last inequality we are using the fact that $g_{i,0}$ satisfies the Morrey-type condition and a covering argument from the fact that $M$ is compact. For II, we have
    \begin{align*}
        \text{II} &= \int_M |V||\hdel u_{i,0}|\left|\frac{d\mu_{g_0}}{d\mu_h} - \frac{d\mu_{g_{i,0}}}{d\mu_h}\right|d\mu_h + \int_M |V||u_{i,0}|\left|\hdel\left(\frac{d\mu_{g_0}}{d\mu_h}-\frac{d\mu_{g_{i,0}}}{d\mu_h}\right)\right|d\mu_h \nonumber \\
        &\leq C_3\int_M |\hdel g_0||\hdel u_{i,0}|\left|\frac{d\mu_{g_0}}{d\mu_h} - \frac{d\mu_{g_{i,0}}}{d\mu_h}\right|d\mu_h + C_4\int_M |\hdel g_0||\hdel g_0 - \hdel g_{i,0}|d\mu_h \nonumber \\
        &\leq C_3\left\lVert \frac{d\mu_{g_0}}{d\mu_h} - \frac{d\mu_{g_{i,0}}}{d\mu_h}\right\rVert_{C^0(M)}\left(\int_M |\hdel g_0|^pd\mu_h\right)^\frac{1}{p} \nonumber \\
        &\qquad \times\left(\int_M |\hdel u_{i,0}|^2d\mu_h\right)^\frac{1}{2}\text{Vol}_h(M)^\frac{p-2}{p} \nonumber \\
        &\quad + C_4\left(\int_M |\hdel g_0 - \hdel g_{i,0}|^pd\mu_h\right)^\frac{1}{p}\left(\int_M |\hdel g_0|^pd\mu_h\right)^\frac{1}{p}\text{Vol}_h(M)^\frac{p-2}{p} \nonumber \\
        &\leq C_3\left\lVert \frac{d\mu_{g_0}}{d\mu_h} - \frac{d\mu_{g_{i,0}}}{d\mu_h}\right\rVert_{C^0(M)} + C_4\left\lVert g_{i,0} - g_0\right\rVert_{W^{1,p}(M)}.
    \end{align*}
    For III, we have
    \begin{align*}
        \text{III} &= \int_M |F_i u_{i,0} - F u_{i,0}|\left|\frac{d\mu_{g_{i,0}}}{d\mu_h}\right|d\mu_h \leq C_5 \int_M |F_i - F|d\mu_h \nonumber \\
        &\leq C_5 \left(\int_M |F_i - F|^\frac{p}{2}d\mu_h\right)^\frac{2}{p}\text{Vol}_h(M)^\frac{p-2}{p} \leq C_5\left\lVert g_{i,0} - g_0\right\rVert_{W^{1,p}(M)}^2.
    \end{align*}
    Finally for IV, we have
    \begin{align*}
        \text{IV} &= \int_M |F u_{i,0}|\left|\frac{d\mu_{g_{i,0}}}{d\mu_h} - \frac{d\mu_{g_0}}{d\mu_h}\right|d\mu_h \nonumber \\
        &\leq C_6\left\lVert \frac{d\mu_{g_0}}{d\mu_h} - \frac{d\mu_{g_{i,0}}}{d\mu_h}\right\rVert_{C^0(M)}\left(\int_M |\hdel g_0|^pd\mu_h\right)^\frac{1}{p}\text{Vol}_h(M)^\frac{p-2}{p} \nonumber \\
        &\leq C_6\left\lVert \frac{d\mu_{g_0}}{d\mu_h} - \frac{d\mu_{g_{i,0}}}{d\mu_h}\right\rVert_{C^0(M)}.
    \end{align*}

    Combining all of the above, we get
    \begin{equation*}
        \left|\langle R_{g_{i,0}}-a, u_{i,0}\rangle - \langle R_{g_0} - a, u_{i,0}\rangle  \right| \leq Cb_i
    \end{equation*}
    for some constant $C = C(n,k_1,k_2,\Lambda,p,\delta,\tilde{u},a,L,\text{Vol}_h(M))$ and $b_i$ is a positive function of $i$ which satisfies $\lim\limits_{i\to\infty} b_i = 0$. Then combining this with (\ref{eqn:scalar-monotonicity}) above, we obtain
    \begin{equation*}
        \int_M (R_{g_i(t)}-a)\tilde{u}d\mu_{g_i(t)} \geq \int_M (R_{g_{i,0}}-a)u_{i,0}d\mu_{g_{i,0}} \geq -Cb_i.
    \end{equation*}
    Finally, observing that $g_i(t)$ converges smoothly to $g(t)$ as $i \to \infty$ for each $t>0$, we let $i \to \infty$ in the inequality above to obtain the desired result.
\end{proof}

Together with Proposition~\ref{prop:distributional-scalar-curvature}, we have a removable singularity result as an application.

\begin{corollary}\label{thm:non-cpt-dist-scalar-curvature-lower-bdd}
Suppose $(M^n,h)$ is a closed manifold and $g_0$ is a $C^0\cap W^{1,p}$ metric on $M$ satisfying the assumption in Proposition~\ref{prop:distributional-scalar-curvature} with $a=0$. If the Yamabe invariant of $M$ is non-positive or equivalently $M$ does not admit metric with positive scalar curvature, then $g_0$, then $g_0$ is Ricci-flat outside $\Sigma$.
\end{corollary}
\begin{proof}
By Proposition~\ref{prop:distributional-scalar-curvature} and Theorem~\ref{thm:preservation-distributional-scalar-curvature-lower-bound}, there exists a Ricci Deturck $h$-flow on $M$ starting from $g_0$ in sense of  $C^0(M)\cap C^\infty_{loc}(M\setminus\Sigma)$ so that $R_{g(t)}\geq 0$.  It follows from the strong maximum principle that $\mathrm{Ric}_{g(t)}\equiv 0$ for all $t>0$. The result follows from smooth convergence away from $\Sigma$ by letting $t\to0$.
\end{proof}

\begin{remark}    
In case the metric $g$ is Lipschitz across the singularity, we are free to choose $\delta=p$. Consider the case where the singularity is in the form of a hypersurface, it was well-known that without additional assumption on mean curvature, the scalar curvature rigidity is in general false while Miao \cite{miao_positive_2003} and Shi-Tam \cite{shi_scalar_2016} showed that the mean curvature inequality is indeed necessary and sufficient. With this in mind, we see that $\delta'<\delta$ is necessary.
\end{remark}

\printbibliography[title=References]

\end{document}